\documentclass[12pt]{article}
\usepackage{mathrsfs}
\usepackage{amsthm}
\usepackage{amssymb}
\usepackage{latexsym}
\usepackage{amsmath,amsfonts}
\usepackage{mathrsfs}
\usepackage{cases}
\usepackage{latexsym,bm}
\usepackage{indentfirst}
\usepackage{color}
\usepackage{ifpdf}
\usepackage{graphicx}
\usepackage{psfrag}
\usepackage[
pdfauthor={CSY},
pdftitle={Strong generalized Halin graphs},
pdfstartview=XYZ,
bookmarks=true,
colorlinks=true,
linkcolor=blue,
urlcolor=blue,
citecolor=blue,
bookmarks=true,
linktocpage=true,
hyperindex=true
]{hyperref}

\usepackage{graphicx}
\usepackage{color}

\title{Average degrees of edge-chromatic critical graphs}

\author{Yan Cao$^a$,
 Guantao Chen$^a$, 
 Suyun Jiang$^b$\thanks{Partially supported by NSFC of China (Nos. 11671232, 11571096, 61373019, 11671186).}, 
 Huiqing Liu$^{c*}$, 
 Fuliang Lu$^{d*}$
 \unskip\\[.5em]
{\small $^a$  Department of Mathematics and Statistics, Georgia State University, Atlanta, GA 30303}\\
{\small $^b$  School of Mathematics, Shandong University, Jinan, 250100}\\
{\small $^c$  Faculty of Mathematics and Statistics, Hubei University, Wuhan 430062}\\
{\small $^d$ School of  Mathematics and Statistics, Linyi University, Linyi, Shandong 276000}\\
}

\date{}

\newtheorem{lem}{Lemma}
\newtheorem{thm}{Theorem}

\newtheorem{conj}{Conjecture}

\newtheorem{cla}{Claim}[section]

\newcommand{\D}{\Delta}
\newcommand{\avd}{\overline{d}}
\newcommand{\phiv}{\varphi}
\newcommand{\phibar}{\bar{\varphi}}

\begin{document}
\newcommand{\udots}{\mathinner{\mskip1mu\raise1pt\vbox{\kern7pt\hbox{.}}
\mskip2mu\raise4pt\hbox{.}\mskip2mu\raise7pt\hbox{.}\mskip1mu}}
\maketitle

\begin{abstract}
Given a graph  $G$, denote by  $\D$, $\avd$  and $\chi^\prime$  the maximum degree, the average degree  and
the  chromatic index of $G$, respectively.   A simple graph $G$ is called  {\it edge-$\Delta$-critical}   if $\chi^\prime(G)=\Delta+1$  and $\chi^\prime(H)\le\Delta$  for every proper subgraph $H$ of $G$.  Vizing in 1968 conjectured that if $G$ is edge-$\D$-critical, then   $\avd \ge \D-1+ \frac{3}{n}$.   We show that
$$
\begin{displaystyle}
\avd \ge
\begin{cases}
 0.69241\D-0.15658 \quad\,\: \mbox{ if } \Delta\geq 66, \\
 0.69392\D-0.20642\quad\;\,\mbox{ if } \Delta=65, \mbox{ and  } \\
 0.68706\D+0.19815\quad\! \quad\mbox{if } 56\leq \Delta\leq64.
 \end{cases}
 \end{displaystyle}
 $$
This result  improves the best known bound $\frac{2}{3}(\Delta +2)$ obtained by Woodall in 2007 for $\D \ge 56$.  Additionally, Woodall constructed an infinite family of graphs showing his result cannot be improved by well-known Vizing's Adjacency Lemma and other known edge-coloring techniques.  To over come the barrier,  we follow the recently developed recoloring technique of Tashkinov trees to expand Vizing fans technique to a larger class of trees.
\end{abstract}

\par {\small {\it Keywords: }\ \  edge-$k$-coloring;  edge-critical graphs;  Vizing's Adjacency Lemma }

\vskip 0.2in \baselineskip 0.1in
\section{Introduction }

All graphs in this paper, unless otherwise stated,  are simple graphs.  Let $G$ be a graph with vertex set $V(G)$ and edge set $E(G)$.  Denote by $\D$ the maximum degree of $G$.    An edge-$k$-coloring of a graph $G$ is a mapping $\varphi:E(G)\rightarrow \{1,2,\cdots,k\}$ such that  $\phiv(e)\neq\phiv(f)$ for any two adjacent edges $e$ and $f$.  We call $\{1, 2, \cdots, k\}$ the color set of $\phiv$.  Denote by $\mathcal{C}^k(G)$ the set of all edge-$k$-colorings of $G$. The chromatic index $\chi^\prime(G)$ is the least integer $k\geq 0$ such
that $\mathcal{C}^k(G)\neq \emptyset$.  We call $G$ {\it class one} if $\chi^\prime(G)=\Delta$. Otherwise, Vizing' theorem \cite{viz1} gives
$\chi^\prime(G)=\Delta+1$ and $G$ is said to be of {\it class two}.
An edge $e$  is called {\it critical} if
$\chi^\prime(G-e)<\chi^\prime(G)$, where $G-e$ is the subgraph obtained from $G$ by removing the edge $e$.
A graph $G$ is called {\it edge-$\D$-critical} if  $\chi^\prime(G) = \D+1$ and $\chi^\prime(H) \le \D$ holds for any proper subgraph $H$ of $G$.  Clearly, if $G$ is edge-$\D$-critical, then $G$ is connected and $\chi^\prime(G-e) = \D$ for any $e\in E(G)$.  Let $\avd(G)$ denote the
average degree of a graph $G$.
Vizing \cite{viz3} made the following conjecture
in 1968, which is commonly referred as Vizing's Average Degree Conjecture.
\begin{conj} {\em [Vizing \cite{viz3}]} \label{conj1} If $G$ is an edge-$\Delta$-critical graph of $n$ vertices, then $\avd(G)\geq
\Delta-1+\frac{3}{n}$.
\end{conj}

The conjecture has been verified   for   $\Delta\leq6$, see \cite{fw,ja,ka,lmz}.
For arbitrary $\Delta$, there are a few results on the  lower bound for $\bar{d}(G)$. Let $G$ be
an edge-$\D$-critical graph.
Fiorini \cite{fs} showed, for $\Delta\geq2$,
$$ \bar{d}(G)\geq
\begin{cases}
\frac{1}{2} (\Delta+1)   \quad     \Delta~\mbox{is odd;}\\
\frac{1}{2} (\Delta+2)        \quad    \Delta~\mbox{is even.}
\end{cases} $$
Haile \cite{hd}  improved the bounds as follows.
$$ \bar{d}(G)\geq
\begin{cases}
\frac{3}{5} (\Delta+2)   \quad\quad\quad      \Delta=9,11,13;\\
\frac{\Delta+6}{2}-\frac{12}{\Delta +4}    \quad\quad \,     \Delta\geq 10,~\Delta~\mbox{is even};\\
\frac{15+\sqrt{29}}{2} \qquad\quad\quad \Delta=15;\\
\frac{\Delta+7}{2}-\frac{16}{\Delta +5}   \quad\quad\:     \Delta\geq 17, \Delta~\mbox{is odd}.
\end{cases} $$
Sanders and Zhao \cite{sz}  showed $\bar{d}(G)\geq\frac{1}{2}(\Delta+\sqrt{2\Delta-1})$ for $\Delta\geq 2.$
 Woodall \cite{w1} improved the bound to $\bar{d}(G)\geq\frac{t(\Delta+t-1)}{2t-1}, \mbox{where} ~t=\lceil\sqrt{\Delta/2} \:\rceil.$
Improving Vizing's Adjacency Lemma, Woodall \cite{w2} improved the coefficient of $\D$ from $\frac 12$ to $\frac 23$ as follows.
$$
\avd(G) \ge
\begin{cases}
\frac{2}{3}(\Delta+ 1) \quad \mbox{ if } \D \ge 2;\\
\frac{2}{3}\Delta+ 1 \quad\;\;\,  \mbox{ if } \D \ge 8;\\
\frac{2}{3}(\Delta+ 2) \quad  \mbox{ if } \D \ge 15.
\end{cases}
$$

In the same paper, Woodall  provided the following  example demonstrating that the above result cannot be improved by the use of  his new adjacency Lemmas (see Lemma \ref{lemma2.4} and Lemma \ref{p}) and Vizing's Adjacency Lemma alone.

Let $G$ be a graph comprising $k$ vertices of degree $4$, all of whose neighbors have degree $\Delta$, and $2k$ vertices of degree $\Delta$, each of which is adjacent to two vertices of degree $4$ and $\Delta-2$ vertices of degree $\Delta$.
Graph $G$ can be chosen to be triangle-free, and indeed to have arbitrarily large girth. Then $G$ may not be edge-$\Delta$-critical, but it satisfies the conclusions of all the existing lemmas at that time including two mentioned above, and it has average degree $\frac{2}{3}(\Delta+ 2)$. So, using these known results, it is impossible to prove that the example is not edge-$\D$-critical.  On the other hand,  we note that using our new result, Claim~\ref{claim 4} in Section 3, it is readily seen that if $\D\ge 6$ then the above example is not edge-$\D$-critical. By proving a few stronger properties of edge-$\D$-critical graphs, we improve Woodall's result as below for $\D \ge 56$.
\begin{thm}\label{main}
If $G$ is an  edge-$\Delta$-critical graph, then
\[
\bar{d}(G) \geq
\begin{cases}
 0.69241\D-0.15658 \quad\,\: \mbox{ if } \Delta\geq 66, \\
 0.69392\D-0.20642\quad\;\,\mbox{ if } \Delta=65, \mbox{ and  } \\
 0.68706\D+0.19815\quad\! \quad\mbox{if } 56\leq \Delta\leq64.
 \end{cases}
 \]
\end{thm}

We will prove a few technic lemmas in Section 2 and give the proof of Theorem~\ref{main} in Section 3.  We will use the following terminology and notation. Let $G$ be a graph.  Denote by  $N(x)$ the neighborhood of $x$ for any $x\in V(G)$,  and $d(x)$ the degree of $x$, i.e., $d(x) = |N(x)|$.
For any nonnegative integer $m$,  we call a vertex $x$  an {\it  $m$-vertex}  if $d(x) = m$, a {\it $(<m)$-vertex}  if $d(x) < m$, and {\it $(>m)$-vertex} if $d(x) >m$.  Correspondingly, we call a
neighbor $y$ of $x$ an {\it  $m$-neighbor}  if $d(y) = m$, etc..  Let $k$ be a positive integer such that $\mathcal{C}^k(G-e) \ne \emptyset$,  and  let  $\varphi\in\mathcal{C}^k(G-e)$ and $v \in V(G)$.   Let $\varphi(v) = \{\varphi(e) \,:\,   e \mbox { is incident with }~ v \}$ and $\bar{\varphi}(v) = \{ 1, \cdots, k\}\setminus\varphi(v)$.
We call $\varphi(v)$ the set of colors seen by  $v$ and $\bar{\varphi}(v)$ the set of colors missing at $v$.
A set $X\subseteq V(G)$
is called {\it elementary} with respect to $\varphi$  if $\bar{\varphi}(u)\cap \bar{\varphi}(v)=\emptyset$ for every two distinct
vertices $u,v \in X$. For any color $\alpha$, let $E_{\alpha}$ denote the set of edges assigned color $\alpha$. Clearly, $E_{\alpha}$ is matching of $G$.  For any two colors $\alpha$ and $\beta$, the components  of induced by edges in $E_{\alpha}\cup E_{\beta}$, named $(\alpha, \beta)$-chains,   are even cycles and paths with alternating color $\alpha$ and $\beta$.
For a vertex $v$ of $G$, we denote by
$P_v(\alpha, \beta, \varphi)$ the unique $(\alpha, \beta)$-chain that contains the
vertex $v$.   Let $\varphi/ P_v(\alpha, \beta, \varphi)$ denote the edge-$k$-coloring obtain from $\phiv$ by switching colors $\alpha$ and $\beta$ on the edges on $P_v(\alpha, \beta, \phiv)$.

\section{Lemmas }

Let $q$ be a positive number, $G$ be an edge-$\D$-critical graph and $x\in V(G)$. For each $y\in N(x)$, let $\sigma_q(x, y) =  |\{ z\in N(y)\setminus \{x\} \ : \ d(z) \ge q\}|$, the number of neighbors of $y$ (except $x$) with degree at least $q$.  Vizing studied the case $q=\D$ and obtained the following result.

\begin{lem} {\em [Vizing's Adjacency Lemma \cite{viz2}]}\label{VAL}
Let $G$ be  an edge-$\Delta$-critical graph.  Then   $\sigma_{\D} (x, y) \ge \D -d(x) +1$ holds for every  $xy\in E(G)$.
\end{lem}

Woodall~\cite{w2} studied  $\sigma_q(x,y)$ for the case $q=2\D -d(x) -d(y) +2$ and obtained the following two results.   For convention, we let $\sigma(x,y) = \sigma_q(x, y)$ when $q=2\D -d(x) -d(y) +2$.
\begin{lem}\label{lemma2.4}{\em [Woodall~\cite{w2}]} Let $xy$ be an edge in an  edge-$\Delta$-critical graph $G$. Then there are at least $\Delta-\sigma(x,y)\geq \Delta-d(y)+1$ vertices $z\in N(x)\setminus\{y\}$ such that $\sigma(x,z)\geq 2\Delta-d(x)-\sigma(x,y).$
\end{lem}
Let $x$ be a vertex in a graph $G$ and $y\in N(x)$.   Vizing's Adjacency Lemma shows that $\sigma_{\D}(x,y) \ge \D -d(x) +1$.  So, $\sigma(x, y) \ge \D - d(x) +1$.  Woodall studied their difference through the following two parameters.
\begin{eqnarray*}
p_{min}(x)& := & \min_{y\in N(x)}\sigma(x,y)-\Delta+d(x)-1 \quad  \mbox{ and }\\
p (x) & := & \min \{\ p_{min}(x), \left\lfloor\frac{d(x)}{2}\right\rfloor-1\ \}.
\end{eqnarray*}
Clearly, $p(x) < d(x)/2 -1$. As a corollary, the following lemma shows that there are about $d(x)/2$ neighbors $y$ of $x$ such that $\sigma(x, y) \ge \D/2$. In general, for any positive number $q$ with $q\le \D$,  we define the following two parameters.
\begin{eqnarray*}
p_{min}(x,q)& := & \min_{y\in N(x)}\sigma_q(x,y)-\Delta+d(x)-1 \quad  \mbox{ and }\\
p (x,q) & := & \min \{\ p_{min}(x,q), \left\lfloor\frac{d(x)}{2}\right\rfloor-3\ \}.
\end{eqnarray*}

\begin{lem}\label{p} {\em [Woodall~\cite{w2}]} Every vertex $x$ in an edge-$\Delta$-critical graph has
 at least $d(x)-p(x)-1$ neighbors $y$ for which $\sigma(x,y)\geq \Delta-p(x)-1$.
\end{lem}

When $d(x)\le \D -4$, we generalize  the above results by allowing  $q$ taking various values and obtain the following two results, which serve as key ingredients in our proof of Theorem~\ref{main}.

\begin{lem}\label{ppp} Let $xy$ be an edge in an  edge-$\Delta$-critical graph $G$ and
$q$ be a positive number.   If  $\D/2 < q \le \D -d(x)/2 -2$,   then
$x$ has at least
 $\Delta-\sigma_q(x,y)-2$ vertices $z\in N(x)\setminus\{y\}$ such that   $\sigma_q(x,z)\geq 2\Delta-d(x)-\sigma_q(x,y)-4$.
\end{lem}

Due to its length, the proof of Lemma~\ref{ppp} will be placed at the end of this section. The following is a consequence of it.
\begin{lem}\label{pp} Let $G$ be an edge-$\D$-critical graph, $x\in V(G)$ and $q$ be a positive number.  If $\D/2 < q \le \D - d(x)/2 -2$,  then $x$ has at least $d(x)-p(x,q)-3$ neighbors $y$ for which $\sigma_q(x,y)\geq \Delta-p(x,q)-5$.
\end{lem}

\begin{proof} Let $y\in N(x)$ such that $p_{min}(x,q)=\sigma_q(x,y)-\Delta+d(x)-1$.

 If $p(x,q) =p_{min}(x,q)$,   by Lemma~\ref{ppp}, $x$ has at least $\Delta-\sigma_q(x,y)-2=\Delta-(\Delta-d(x)+p_{min}(x,q)+1)-2 =
 d(x)-p_{min}(x,q)-3$ vertices $z\in N(x)\setminus\{y\}$ such that $\sigma_q(x,z)\geq 2\Delta-d(x)-4-\sigma_q(x,y)=\Delta-p_{min}(x,q)-5$.

If
$p(x,q)=\left\lfloor\frac{d(x)}{2}\right\rfloor-3<p_{min}(x,q)$, then for every $y\in N(x),$ $\sigma_q(x,y)>\Delta -d(x)+1+\left\lfloor\frac{d(x)}{2}\right\rfloor-3
\geq\Delta-\left\lfloor\frac{d(x)}{2}\right\rfloor-3
=\Delta-p(x,q)-6$. So $\sigma_q(x,y)\geq\Delta-p(x,q)-5$.
\end{proof}

Let $G$ be a graph (in this paragraph, $G$ may be a multigraph), $e_1=y_0y_1\in E(G)$ and $\phiv\in C^k(G-e_1)$.
A Tashkinov tree $T$  is a sequence $T=(y_0,e_1,y_1,\cdots,e_p,y_p)$ with $p\geq 1$ consisting of edges $e_1, e_2, \cdots,e_p$ and vertices $y_0,y_1\cdots,y_p$ such that the following two conditions hold.
\begin{itemize}
\item The edges $e_1, e_2, \cdots,e_p$ are distinct and   $e_i=y_ry_{i}$ for each  $1\leq i\leq p$, where $r < i$; \\
\item For every edge $e_i$ with $2\leq i\leq p,$ there is a vertex $y_h$ with $0\leq h<i$ such that $\varphi(e_i)\in \bar\varphi(y_h).$
\end{itemize}
Clearly, a Tashkinov tree is indeed a tree of $G$. Tashkinov~\cite{Tashkinov-2000} proved that if $G$ is edge-$k$-critical  with $k \ge \D +1$, then $V(T)$ is elementary. In the above definition, if $e_i = y_0y_i$ for every $i$ (i.e., $T$ is a star with $y_0$ as the center),  then $T$ is a Vizing fan. The classic result of Vizing~\cite{ss} show that for every Vizing fan $T$ the set $V(T)$ is elementary if $G$ is edge-$k$-critical for every $k \ge \D$, which  includes  edge-$\D$-critical graphs. In the definition of Tashkinov tree, if $e_i = y_{i-1}y_i$ for every $i$ (i.e. $T$ is a path with end-vertices $y_0$ and $y_p$), then $T$ is a Kierstead path, which was introduced by Kierestead~\cite{Kierstead84}.  Kierstead proved that for every Kierstead path $P$ the set $V(P)$ is elementary if $G$ is an edge-$k$-critical with $k \ge \D +1$.  For simple graphs, following Kierstead's proof, Zhang~\cite{Zhang2000} noticed that for a Kierstead path $P$ the set $V(P)$ is elementary if $G$ is edge-$\D$-critical and $d(y_i) < \D$ for every $i$ with  $2\le i \le p$.
Clearly,  every Kierstead path $P$ with three vertices is a Vizing fan, so $V(P)$ is elementary if $G$ is edge-$\D$-critical.

\begin{lem}\label{p4}{\em [Kostochka and Stiebitz~\cite{ss}]}
Let $G$ be a graph with maximum degree $\Delta$ and ${\chi}^\prime(G)= {\Delta}+1.$ Let $e_1\in E(G)$ be a critical edge and $\varphi\in\mathcal{C}^\Delta(G-e_1).$ If $K=(y_0,e_1,y_1,e_2,y_2,e_3,y_3)$ is a Kierstead path with respect to $e_1$ and $\varphi,$ then the following statements hold:
\begin{enumerate}
\item  $\bar{\varphi}(y_0)\cap \bar{\varphi}(y_1)=\emptyset;$
\item if $d(y_2)<\Delta,$ then $V(K)$ is elementary with respect to $\varphi;$
\item  if $d(y_1)<\Delta,$ then $V(K)$ is elementary with respect to $\varphi;$
\item  if $\Gamma=\bar{\varphi}(y_0)\cup \bar{\varphi}(y_1),$ then $|\bar{\varphi}(y_3)\cap \Gamma|\leq 1.$
\end{enumerate}
\end{lem}

In the definition of Tashkinov tree $T=(y_0, e_1, y_1, e_2, y_2, \cdots,  y_p)$,  we call $T$   a {\it broom} if $e_2=y_1y_2$ and for each $i \ge 3$, $e_i=y_2y_i$, i.e., $y_2$ is one of the end-vertices of $e_i$ for each $i \ge 3$. Moreover, we call $T$ a {\it simple broom} if $\phiv(e_i) \in \phibar(y_0)\cup\phibar(y_1)$ for each $i\ge 3$, i.e., $(y_0, e_1, y_1, e_2, y_2, e_i, y_i)$ is a Kierstead path.

\begin{lem}\label{broom}{\em [Chen, Chen, Zhao~\cite{cchen}]}
Let $G$ be an edge-$\D$-critical graph, $e_1=y_0y_1 \in E(G)$ and $\phiv \in\mathcal{C}^\Delta(G-e_1)$ and $B=\{y_0, e_1, y_1, e_2, y_2, \cdots, e_p,  y_p\}$ be a simple broom. If $|\bar{\varphi}(y_0)\cup \bar{\varphi}(y_1)|\geq 4$ and $\min\{d(y_1),d(y_2)\}< \Delta,$  then $V(B)$ is elementary with respect to $\varphi$.
\end{lem}

\begin{lem}\label{lemfact} Let $G$ be an edge-$\Delta$-critical graph,
$xy\in E(G)$, and  $\varphi\in\mathcal{C}^{\Delta}(G-xy)$.
Let $q$ be a positive number with   $d(x)<q\leq\Delta-1$ and  $Z=\{z\in N(x)\backslash \{y\}\ : \ d(z)>q, \varphi(xz)\in \bar{\varphi}(y)\}$.
Then for every $z\in Z$ the following three inequalities hold.
\begin{eqnarray}
&& |Z|  \geq  \Delta-d(y)+1-\left\lfloor\frac{d(x)+d(y)-\Delta-2}{\Delta-q}\right\rfloor\\
&&  \sum_{z\in Z}(d(z)-q)  \geq  (\Delta-d(y)+1)(\Delta-q)-d(x)-d(y)+\Delta+2\\
&&  \sigma_{q}(x,z)  \geq  2\Delta-d(x)-d(y)+1-\left\lfloor\frac{d(x)+d(y)+d(z)-2\Delta-2}{\Delta-q}\right\rfloor
  \end{eqnarray}
  \end{lem}

\begin{proof}
Since $xy$ is a critical edge of $G$,   $\bar{\varphi}(x)\cap \bar{\varphi}(y)=\emptyset$.  Let $Z_y := \{z \in N(x) \setminus \{y\}\ : \  \phiv(xz) \in \phibar(y)\}$.  Clearly, $Z \subseteq Z_y$ and $|Z_y| = \D -d(y) +1$.  Since $\{y, x\}\cup Z_y$ forms a Vizing fan with center $x$,    it is elementary, so
$|\phibar(x)| +  |\phibar(y)| + \sum_{z\in Z_y}|\bar{\varphi}(z)|\leq \Delta$ holds.  Since $|\phibar(x)|  = \D -d(x) +1$ and $|\phibar(y)| = \D -d(y) +1$,  we have
\begin{equation}\label{Zy}
\sum_{z\in Z_y}|\bar{\varphi}(z)| \le \D - |\phibar(x)| - |\phibar(y)| \le d(x) + d(y) -\D -2.
\end{equation}
Since $d(z) \le q$ for all $z\in Z_y-Z$, $\sum_{z\in Z_y}|\bar{\varphi}(z)|  \geq(|Z_y|-|Z|)(\Delta-q)$.   Solving for $|Z|$, we get
$|Z|\geq |Z_y| -\left\lfloor\frac{d(x)+d(y)-\Delta-2}{\Delta-q}\right\rfloor.$
Since $|Z_y| = \D -d(y) +1$, inequality (1) holds.

Plugging $|\phibar(z)| = \D -d(z)$ for each $z\in Z_y$ in inequality (\ref{Zy}), we get
\[
\sum_{z\in Z_y} d(z) \ge |Z_y| \D -(d(x) + d(y) -\D -2).
\]
Since $d(z) \le q$ for every $z\in Z_y -Z$, we have
\[
 \sum_{z\in Z}(d(z) - q) \ge \sum_{z\in Z_y}(d(z) -q)
 \ge  |Z_y| \D -(d(x) +d(y) -\D -2) - |Z_y|q.
 \]
 \\
Plugging $|Z_y| =  \D - d(y) +1$, we get (2).

For each $z\in Z$,  let $U^*_z=\{ u\in N(z)\backslash\{x\}\ : \  \varphi(zu)\in \bar{\varphi}(x)\cup\bar{\varphi}(y)\backslash\{\varphi(xz)\} \}$ and
$U_z = \{ u\in U^*_z \ : \ d(u) > q\}$.  Clearly, $|U^*_z| = 2\D -d(x) -d(y) +1$ and $\{y, x, z\}\cup U^*_z$ forms a simple broom.
 Since $d(x)<q\leq \Delta-1$, we have $d(x)\leq \Delta-2$. Thus $|\bar{\varphi}(x)\cup\bar{\varphi}(y)|\geq 4$ and $\min\{d(x),d(z)\}=d(x)<\Delta$.  By Lemma \ref{broom},
 $\{y,x,z\}\cup U^*_z$ is elementary with respect to $\varphi$. So
 \[
 \sum_{u\in U^*_z} |\bar{\varphi}(u)| + |\bar{\varphi}(x)| + |\bar{\varphi}(y)|+|\bar{\varphi}(z)|\leq \Delta,
 \]
  which in turn gives
$\sum_{u\in U^*_z}|\bar{\varphi}(u)| \le d(x)+d(y)+d(z)-2\Delta-2$.
Since $d(u) \le q$ for every $u\in U^*_z - U_z$,
$\sum_{u\in U_z^*} |\phibar(u)| \ge (|U_z^*| -|U_z|)(\Delta-q)$.
So,
\[
(|U_z^*| -|U_z|)(\Delta-q) \le d(x)+d(y)+d(z)-2\Delta-2.
\]
Solving the above inequality with $|U_z^*| = 2\D -d(x) -d(y) +1$,
we get
\[
|U_z| \ge  2\D -d(x) -d(y) +1 - \left\lfloor\frac{d(x)+d(y)+d(z)-2\Delta-2}{\Delta-q}\right\rfloor.
\]
Since $\sigma_q(x,z) \ge |U_z|$, the inequality (3) holds.
 \end{proof}

\subsection{Proof of Lemma \ref{ppp}}

\noindent \textbf{Lemma \ref{ppp}.}
Let $xy$ be an edge in an  edge-$\Delta$-critical graph $G$ and
$q$ be a positive number.   If  $\D/2 < q \le \D -d(x)/2 -2$,   then
$x$ has at least
 $\Delta-\sigma_q(x,y)-2$ vertices $z\in N(x)\setminus\{y\}$ such that   $\sigma_q(x,z)\geq 2\Delta-d(x)-\sigma_q(x,y)-4$.

\begin{proof}
Let graph $G$, edge $xy\in E(G)$ and $q$ be defined as  in Lemma~\ref{ppp}.  A neighbor $z\in N(x)\setminus \{y\}$ is called {\it feasible} if there exits a coloring $\phiv\in \mathcal{C}^{\D}(G-xy)$ such that $\phiv(xz) \in \phibar(y)$, and such a coloring $\phiv$ is called {\it $z$-feasible}. Denote by
 $\mathcal{C}_z$ the set of all $z$-feasible colorings.
 For each  $\phiv\in \mathcal{C}_z$.  let
\begin{eqnarray*}
Z(\phiv)  & = & \{v\in N(z)\setminus\{x\} \,:\,  \varphi(vz)\in \bar{\varphi}(x)\cup \bar{\varphi}(y)\}, \\
C_z(\phiv) & = & \{\varphi(vz) \,:\, v\in Z(\phiv) \mbox{ and } d(v) < q\}, \\
Y(\phiv) & =& \{v\in N(y)\setminus\{x\}\,:\,  \varphi(vy)\in \bar{\varphi}(x)\cup \bar{\varphi}(z)\}, \mbox{ and } \\
C_y(\phiv)  & = & \{\varphi(vy) \,:\,  v\in Y(\phiv) \mbox{ and } d(v) < q\}.
\end{eqnarray*}
Note that $Z(\phiv)$ and $Y(\phiv)$ are vertex sets while
$C_z(\phiv)$ and $C_y(\phiv)$ are color sets.
For each color $k\in \phiv(z)$, let $z_k \in N(z)$ such that $\phiv(zz_k) = k$. Similarly, we define $y_k$ for each $k\in \phiv(y)$. Let
$T(\phiv) = \{k \in \phiv(x)\cap\phiv(y)\cap \phiv(z) \,:\,
d(y_k) < q \mbox{ and } d(z_k) < q \}.$

Since $G$ is edge-$\D$-critical, $\{x, y, z\}$ is elementary with respect to $\phiv$.  So $\phibar(x)$, $\phibar(y)$, $\phibar(z)$ and $\phiv(x)\cap \phiv(y)\cap \phiv(z)$ are mutually exclusive and
$$
\phibar(x)\cup \phibar(y) \cup \phibar(z)\cup (\phiv(x)\cap\phiv(y)\cap \phiv(z)) = \{1, 2, \dots, \D\}.
$$
Recall that  $\sigma_q(x,y)$ and $\sigma_q(x,z)$ are number of vertices with degree $\ge q$ in $N(y)\setminus\{x\}$ and $N(z)\setminus \{x\}$, respectively. So, the following inequalities hold.
\begin{eqnarray*}
&&\sigma_q(x,y) + \sigma_q(x,z) \\
& \ge & |Y(\phiv)|-|C_y(\phiv)| + |Z(\phiv)| -|C_z(\phiv)| + |\phiv(x)\cap \phiv(y)\cap\phiv(z)| -|T(\phiv)|\\
& = & |\phibar(x)\cup\phibar(z)| + |\phibar(x)\cup\phibar(y)|-1 +
|\phiv(x)\cap\phiv(y)\cap \phiv(z)| - |C_y(\phiv)| -|C_z(\phiv)| -|T(\phiv)|\\
& = & \D + |\phibar(x)| - |C_y(\phiv)| -|C_z(\phiv)| -|T(\phiv)|-1\\
& = &
2\D -d(x) +1 -|C_y(\phiv)| -|C_z(\phiv)| -|T(\phiv)|-1
\end{eqnarray*}
So, Lemma~\ref{ppp} follows the three statements below.
\begin{itemize}
\item [{\bf I.}] For any $\phiv\in \mathcal{C}_z$, $|C_z(\phiv)| \le 1$ and $|C_y(\phiv)| \le 1$;
\item [{\bf II.}] there exists a $\phiv\in \mathcal{C}_z$ such that
$|T(\phiv)| \le 2$; and
\item [{\bf III.}] there are $\D -\sigma_q(x,y) -2$ feasible vertices $z\in N(x)\setminus\{y\}$.
\end{itemize}

For every $z$-feasible coloring $\phiv\in \mathcal{C}^{\D}(G-xy)$, let $\phiv^d\in \mathcal{C}^{\D}(G-xz)$ obtained from $\phiv$ by assigning $\phiv^d(xy) = \phiv(xz)$ and keeping all colors on other edges unchange.  Clearly, $\phiv^d$ is a $y$-feasible coloring and $Z(\phiv^d)=Z(\phiv)$, $Y(\phiv^d) = Y(\phiv)$, $C_z(\phiv^d) =C_z(\phiv)$ and $C_y(\phiv^d) = C_y(\phiv)$.  We call $\phiv^d$ the {\it dual coloring} of $\phiv$. Considering dual colorings, we see that some properties for vertex $z$  also hold for vertex $y$.

The condition $q \le \D -d(x)/2 -2$ implies
$2(\D-q) + (\D -d(x)) +1 > \D$. So, for any $\phiv\in \mathcal{C}^{\D}(G-xy)$,  every elementary set $X$ with  $x\in X$ contains at most one vertex with degree $\le q$.

Let $z\in N(x)\setminus\{y\}$ be a feasible vertex and $\phiv\in \mathcal{C}_z$. By the definition of $Z(\phiv)$,  $G[\{x, y, z\}\cup Z(\phiv)]$ contains a simple broom, so $\{x, y, z\}\cup Z(\phiv)$ is elementary with respect to $\phiv$. Consequently, it contains at most one vertex other than $x$ having degree $< q$. Thus, $|C_z(\phiv)| \le 1$.  By considering its dual $\phiv^d$, we have $|C_y(\phiv)| =|C_y(\phiv^d)|\le 1$. Hence, {\bf I} holds. The proofs of {\bf II} and {\bf III} are much more complicated.
In the remainder of the proof, we let $Z=Z(\phiv)$, $Y=Y(\phiv)$, $C_z = C_z(\phiv)$,  $C_y = C_y(\phiv)$, and $T = T(\phiv)$ if the coloring $\phiv$ is clearly referred. Let $R=C_z\cup C_y$ and $\phibar(x,R)= \phibar(x)\setminus R$. A  coloring $\phiv\in \mathcal{C}_z$ is called {\it optimal} if $|C_z|+|C_y|$ is maximum over all feasible colorings.

\subsubsection{Proof of {\bf II}.}
Suppose to the contrary: $|T| \ge 3$ for every $\phiv\in \mathcal{C}_z$.  Let $\phiv$ be an optimal feasible coloring and assume, without loss of generality, $\phiv(xz) =1$.

{\flushleft \bf Claim A.} For each $i\in \phibar(x,R)$ and
$k\in T$,  $P_x(i,k,\varphi )$ contains both $y$ and $z$.

\begin{proof} We first show that $z\in V(P_x(i,k,\varphi ))$. Otherwise, $P_z(i, k, \phiv)$ is disjoint with $P_x(i, k, \phiv)$. Let $\varphi^\prime=\varphi / P_z(i,k,\varphi )$.  Since $1 \ne i, k$, $\phiv^\prime$ is also feasible. Since colors in $R$ are unchanged and $d(z_k)< q$, $C_z(\phiv^\prime) = C_z\cup \{i\}$ and $C_y(\phiv^\prime) \supseteq C_y$, giving a contradiction to
the maximality of $|C_y| + |C_z|$.  By considering the dual $\phiv^d$, we can verify that $y\in V(P_x(i,k,\varphi ))$.
\end{proof}

Since $|T| \ge 3$, there are three colors  $k_1,k_2,k_3\in T$. Let \begin{eqnarray*}
V_T & = &\{z_{k_1},z_{k_2},z_{k_3}\}\cup\{y_{k_1},y_{k_2},y_{k_3}\},\\
W(\phiv) & = & \{u\in V_T\,:\, \phibar(u)\cap \phibar(x)\subseteq  R \},\\
M(\phiv) & = & V_T-W(\phiv)=\{u\in V_T \,:\, \phibar(u)\cap\phibar(x,R)\neq \emptyset\},\\
E_T & = & \{zz_{k_1}, zz_{k_2}, zz_{k_3}, yy_{k_1},
yy_{k_2}, yy_{k_3}\},\\
E_W(\phiv) & = & \{e\in E_T \,:\,  \mbox{ $e$ is incident to a vertex in $W(\phiv)$}\},\mbox{ and } \\
E_M(\phiv) & = & E_T-E_W(\phiv)=\{e\in E_T\,:\, \mbox{ $e$  is incident to a vertex
in $M(\phiv)$} \}.
\end{eqnarray*}
For convenience, we let $W=W(\phiv),M=M(\phiv),E_W=E_W(\phiv)$ and $E_M=E_M(\phiv)$.

We assume that $|E_W|$ is minimum over all optimal feasible coloring $\phiv$ and all sets of three colors in $T(\phiv)$. For each $v\in M$, pick a color $\alpha_v\in \phibar(v)\cap \phibar(x,R)$. Let $C_M=\{\alpha_v \,:\, v\in M\}$. Clearly, $|C_M| \le |M|$.  Note that $\{z_{k_1},z_{k_2},z_{k_3}\}\cap \{y_{k_1},y_{k_2},y_{k_3}\}$ may be not empty, $\frac{|E_W|}{2}\leq |W|\leq |E_W|$ and $\frac{|E_M|}{2}\leq |M|\leq |E_M|$.

{\bf Claim B.}
If there exist two vertices $u,v\in V_T$ and a color
 $\alpha \in \phiv(x)\setminus R$ such that
$\alpha \in \phibar(u)\cap \phibar(v)$, then there is an optimal feasible coloring $\phiv^*$ such that  $|E_W(\phiv^*)| \le |E_W|$ and $\{u, v\}\cap M(\phiv^*)\ne \emptyset$. Moreover, if $\phibar(x)\setminus (R\cup C_M) \ne \emptyset$, then $u$ or $v\in M$.

\begin{proof}
We first note that the condition of $d(x)$ and $q$ gives
\begin{equation}\label{phibar(x)}
|\phibar(x)| = \D - d(x)+1 \ge \D -2(\D -q) +5> 5.
\end{equation}

If $\{u, v\}\cap M \ne \emptyset$, we are done. Suppose $u, v\in W$. Let $\beta$ be an arbitrary color in $\phibar(x,R)$ with the preference that $\beta\in \phibar(x,R)\setminus C_M$ if the set is not empty. Since $|R|\le 2$ and (\ref{phibar(x)}), such a color $\beta$ exists. Since $u, v\in W$, we have $\beta \in \phiv(u)\cap \phiv(v)$. So, both $u$ and $v$ are endvertices of $(\alpha, \beta)$-chains. Assume without loss of generality $P_u(\alpha, \beta, \phiv)$ is disjoint with $P_x(\alpha, \beta, \phiv)$. We note that $\beta\in \phiv(y)\cap \phiv(z)$
since $\{x, y, z\}$ is an elementary set.

We first consider the case of $\alpha=1$. In this case,   $P_{x}(\alpha,\beta,\varphi)=P_{y}(\alpha,\beta,\varphi)$ holds; otherwise, $\phiv/P_x(\alpha,\beta, \phiv)$ would lead a $\D$-coloring of $G$.
Since $\phiv(xz)=1$, $z\in P_x(\alpha, \beta, \phiv)$.   So, $P_u(\alpha, \beta, \phiv)\cap \{x, y, z\} =\emptyset$.  Hence, coloring $\varphi^\prime=\varphi/ P_{u}(\alpha,\beta,\varphi)$ is  feasible, $C_y(\phiv^\prime) = C_y$, $C_z(\phiv^\prime) = C_z$ and $T(\phiv^\prime) = T$. So, $\phiv^\prime$ is also optimal, $u\in M(\phiv^\prime)$ and $|E_W(\phiv^\prime)| \le |E_W|$ with that the inequality holds if the other endvertex of $P_u(\alpha, \beta, \phiv)$ is not in $M$ or $\beta\notin R\cup C_M$.

We now suppose  $\alpha\in \varphi (x)\setminus (R \cup \{1\})$. So, both $\alpha$ and $\beta$ are not in $R\cup \{1\}$.
Let $\varphi^\prime=\varphi/ P_{u}(\alpha,\beta,\varphi)$. Then,
$\phiv^\prime$ is feasible (for $z$), $C_y(\phiv^\prime)=C_y$ and $C_z(\phiv^\prime) = C_z$. Thus,  $\varphi^\prime$ is still an optimal coloring and $\beta\in \overline{\varphi^\prime }(u)$. We have $|E_W(\phiv^\prime)| \le |E_W|$ and $u\in M(\phiv^\prime)$. By the minimality of $|E_W|$, we have the other endvertex  of $P_u(\alpha,\beta, \phiv)$ must be in $M$ and $\beta\in C_M$, which leads a contradiction to the minimality of $|E_W|$  if $\beta\notin R\cup C_M$.
\end{proof}

{\bf Claim C.}
There exist a color $k\in \{k_1, k_2, k_3\}$ and three distinct colors $i,j,\ell$ where $i,j\in \phibar(x,R)$ and $\ell\in \phibar(x,R)\cup\{1\}$ such that $i\in \phibar(z_k)$, $j\in \phibar(y_k)$ and $\ell \in \phibar(z_k)\cup \phibar(y_k)$.

\begin{proof} We first note that if there exist $i,j\in \phibar(x, R)$ such that $i\in  \phibar(z_k)$ and $j\in  \phibar(y_k)$, then $i\neq j$; for otherwise, by Claim A, the path $P_{x}(i,k,\varphi )$ contains three endvertices $x,z_k$ and $y_k$, a contradiction.

First we show that there exist $i,j\in \bar{\varphi }(x, R)$ and $k\in\{k_1,k_2,k_3\}$ such that $i\in \bar{\varphi }(z_k)$ and $j\in \bar{\varphi }(y_k)$.  Suppose not.  Then $|E_M|\leq 3$ and $|E_W|\geq 3$, which in turn give $|W|\ge \lceil\frac{3}{2}\rceil=2$ and $|M| \le |E_M|\leq 3$.   Let $u,v\in W$.
By (\ref{phibar(x)}), $|\phibar(x)|\ge 6 > |R|+|M|$. There exists a color
$\beta\in \phibar(x, R)\setminus C_M$.  Then, $\beta\in \varphi (u)\cap\varphi (v)$ as $u,v\in W$.

Since $|R| \le 2$,  we have
$$
|\phibar(u)\setminus R|+|\phibar(v)\setminus R| +|\phibar(x)| >2(\Delta-q -2) + \D -d(x) +1 \ge \D +1.
$$
 So, there is a color $\alpha$ shared by at least two of these three sets. Since $(\phibar(u)\setminus R)\cap \phibar(x) = \emptyset$ and $(\phibar(v)\setminus R)\cap \phibar(x) = \emptyset$, we have
$\alpha \in (\phibar(u)\setminus R)\cap(\phibar(v)\setminus R)\cap \phiv(x)$.
By Claim B, there exists an optimal feasible coloring $\phiv^\prime$ such that $|E_W(\phiv^\prime)| \le |E_W|$. Moreover, since $\beta\notin R\cup C_M$, the inequality holds which gives a contradiction to the minimality of $|E_W|$.

We now only need to show that additionally there exists another color $\ell\in \phibar(x,R)\cup \{1\}$ such that $\ell\in \phibar(y_k)\cup \phibar(z_k)$.  Suppose on the contrary that there is no such a color $\ell$.  Then the following equalities hold.
\[
\phibar(z_k)\cap \phibar(x,R)=\{ i\} \quad \mbox{ and }
\quad \phibar(y_k)\cap \phibar(x,R)=\{ j\}
\]
 Moreover, $1\notin \phibar(z_k)\cup \phibar(y_k)$. Since $|R|\le 2$,   the following inequalities hold.
\[
|\phibar(z_k)\setminus (R\cup\{i\})|+
|\phibar(y_k)\setminus  (R\cup\{j\})| +
|\phibar(x)\cup (R\cup \{1\})|
>2(\Delta-q-|R|-1) + \D -d(x)+2 \geq \D
\]
So, there is color $\alpha$ in two of the three sets. Since
$(\phibar(z_k)\setminus (R\cup\{i\})\cap (\phibar(x)\cup R\cup \{1\}) = \emptyset$ and
$(\phibar(y_k)\setminus (R\cup\{j\}))\cap (\phibar(x)\cup R\cup \{1\}) = \emptyset$, $\alpha \in \phibar(z_k)\cap \phibar(y_k)\cap \phiv(x)\setminus (R\cup\{i, j, 1\})$.

Since $|\phibar(x)| =\D -d(x) +1 \ge \D -2(\D-q) +5 >5$, there exists a color $\beta\in \phibar(x)\setminus (R\cup\{i, j\})$.
Then, $\beta\notin \phibar(z_k)\cup \phibar(y_k)$. So,
$\beta\in (\phibar(x)\cap \phiv(z_k)\cap \phiv(y_k))\setminus (R\cup \{i, j\})$.

Applying  Claim B with color $\alpha$, we obtain an optimal coloring $\phiv^\prime$ and $|E_W(\phiv^\prime)| \le |E_W|$, but color $\beta$ serves as the required color $\ell$, giving a contradiction.
\end{proof}

 Let $k$, $i$, $j$ and $\ell$ be as stated in Claim C.
 If $\ell\neq 1$, we consider  coloring obtained from $\phiv$ by interchange colors $1$ and $\ell$ for edges not on the path $P_x(1, \ell, \phiv)$, and rename it as $\phiv$.  So we may
assume $1\in \phibar(y_k)\cup \phibar(z_k)$.

We first consider the case  of $1\in \phibar(y_k)$. By Claim A, the paths $P_x(i,k,\varphi)$ and $P_x(j,k,\varphi)$ both contain $y,z$.  Since
$\phiv(yy_k)=\phiv(zz_k)=k$,  these two paths   also contain $y_k,z_k$. Since $i\in \bar{\varphi }(z_k)$, we have $x$ and $z_k$ are the two endvertices of $P_x(i, k, \phiv)$. So, $i\in \phiv(y)\cap\phiv(z)\cap\phiv(y_k)$.  Similarly, we have $j\in \phiv(y)\cap \phiv(z)\cap \phiv(z_k)$.
We now consider the following sequence of colorings of $G-xy$.

 Let $\phiv_1$ be obtained from $\phiv$ by assigning $\phiv_1(yy_k) = 1$.  Since $1$ is missing at both $y$ and $y_k$,  $\phiv_1$ is an edge-$\D$-coloring of $G-xy$. Now $k$ is missing at $y$ and $y_k$, $i$ is still missing at $z_k$.  Since $G$ is not $\D$-colorable, $P_x(i,k,  \phiv) = P_y( i, k, \phiv)$; otherwise $\phiv/P_y( i, k, \phiv)$ can be extended to an edge-$\D$-coloring of $G$ giving a contradiction.  Furthermore, $z_k, y_k\notin V(P_x(i, k, \phiv_1))$ since either $i$ or $k$ is missing at these two vertices, which in turn shows that $z\notin V(P_x(i, k, \phiv_1))$ since $\phiv_1(zz_k) =k$.

Let $\varphi _2=\varphi _1/P_x(i,k,\varphi_1 )$. We have $k\in \bar{\varphi _2}(x), i \in \bar{\varphi _2}(y)\cap\bar{\varphi _2}(z_k)$ and $ j\in \bar{\varphi _2}(x)\cap\bar{\varphi _2}(y_k)$. Since $G$ is not edge-$\D$-colorable, $P_x(i, j, \phiv_2) = P_y(i,j, \phiv_2)$ which  contains neither  $y_k$ nor $z_k$.

Let $\phiv_3 = \phiv_2/P_x(i, j, \phiv_2)$.  Then $k\in \bar{\varphi _3}(x)$ and $j\in \bar{\varphi _3}(y)\cap \bar{\varphi _3}(y_k)$.

Let $\varphi _4$ be obtained from $\varphi _3$ by recoloring $yy_k$ by $j$.  Then $1\in \bar{\varphi _4}(y)$, $\varphi _4(xz)=1$, $k\in \bar{\varphi _4}(x)$, $\varphi _4(zz_k)=k$. Since $\phiv_4(xz) = 1\in \phibar_4(y)$,  $\phiv_4$ is feasible.
Since $i,j,k\notin R=C_y\cup C_z$, the colors in $R$ are unchanged during this sequence of re-colorings, so $C_y(\phiv_4) \supseteq C_y$ and $C_z(\phiv_4) \supseteq C_z$.  Since $\phiv_4(zz_k)=k\in \phibar_4(x)$ and $d(z_k) < q$, we have $k =\phiv_4(zz_k)\in C_z(\phiv_4)$. So, $C_z(\phiv_4) \supseteq C_z\cup \{k\}$. We therefore have $|C_y(\phiv_4)| + |C_z(\phiv_4)| \ge |C_y| + |C_z| +1$, giving a contradiction.

For  the case of  $1\in \bar{\varphi }(z_k)$, we consider the dual coloring $\phiv^d$ of $G-xz$ obtained  from $\varphi$ by uncoloring $xz$ and coloring $xy$ with color $1$.  Following the exact  same argument above, we can reach a contradiction to the maximum of $|C_y| + |C_z|$.  This completes the proof of {\bf II}.
 \end{proof}

\subsubsection{Proof of {\bf III}.}

Denote by $Z$ the set of all feasible vertices.  For a coloring $\phiv\in \mathcal{C}^{\D}(G-xy)$, let $Z(\phiv) =\{ z\in N(x)\,:\, \phiv(xz) \in \phibar(y)\}$ and $S(\phiv) = \{z\in N(x)\setminus Z(\phiv) \,:\,
d(y_{\phiv(xz)}) < q\}$, where $y_{j}\in N(y)$ with $\phiv(yy_j) = j$ for any color $j$.    We call vertices in $S(\phiv)$ {\it semi-feasile} vertices of $\phiv$.

\begin{cla}\label{cla4} For any coloring $\phiv\in \mathcal{C}^{\D}(G-xy)$, the following two inequalities hold.
\begin{itemize}
\item[{\bf a.}] $|Z(\phiv)\cup S(\phiv)| \ge \D -\sigma_q(x, y) -1$;

\item [{\bf b.}] \label{-2}  With one possible exception,  for all $z\in  S(\phiv)$ there exists a coloring $\phiv^*\in \mathcal{C}^{\D}(G-xy)$ such that
$\phiv^*(xz) \in \bar{\phiv^*}(y) $.
\end{itemize}
 \end{cla}

\begin{proof} Let $\phiv\in \mathcal{C}^{\D}(G-xy)$.
Since $G$ is edge-$\Delta$-critical, $xy$ is an edge of $G$, it is easy to see that $\bar{\varphi}(y)\subseteq \varphi(x)$ and $\bar{\varphi}(x)\subseteq \varphi(y)$.
Divide $\phiv(y)$ into two subsets:
\[
\phiv(y, \ge q) =\{ i\in \phiv(y) \,:\, d(y_i) \ge q\}\quad  \mbox{ and } \quad
\phiv(y, <q) = \{ i \in \phiv(y)\,:\, d(y_i) < q\}.
\]

Clearly, $\sigma_q(x, y) = |\phiv(y, \ge q)|$ and
$|\phibar(y)| + |\phiv(y, <q)| = \D -\sigma_q(x,y)$. Since $\phibar(y)\subseteq \phiv(x)$, to prove {\bf a}, we only need to show $|\phibar(x)\cap \phiv(y, <q)| \le 1$.
Since  edge $xy$ and the edges incident to $y$ with colors in $\phibar(x)$ form a Vizing fan $F$, the vertex set $V(F)$ is elementary with respect to $\phiv$.
Since $d(x) +2q < 2\D$, $V(F)\setminus \{x\}$ contains at most one vertex with degree $< q$.  So  $|\phibar(x)\cap \phiv(y, <q)| \le 1$ holds.

To prove {\bf b}, we show that  for any two distinct vertices $z_{k},z_{\ell}\in S(\phiv)$, there is a coloring $\phiv^*\in \mathcal{C}^{\D}(G-xy)$ such that at least one of $\phiv^*(xz_k)$ and $\phiv^*(xz_{\ell})$ is in $\bar{\phiv^*}(y)$.  We assume $\phiv(xz_k) = k$ and $\phiv(xz_{\ell}) =\ell$.
Let $y_k, y_{\ell}\in N(y)\setminus \{x\}$ such that $\phiv(yy_k)=k$ and $\phiv(yy_{\ell}) =\ell$.

By the definition of $S(\phiv)$,  we have
 $d(y_{k})<q$ and  $d(y_{\ell})<q$. Since $\D/2 < q \le \D -d(x)/2 -2$, the following inequality holds.
 \begin{equation}\label{III>D}
 |\phibar(x)| + |\phibar(y_k)| + |\phibar(y_{\ell})| > \D
 \end{equation}
We claim that there exists a coloring $\phiv^*\in  \mathcal{C}^{\Delta}(G-xy)$ such that keeping the property  $\phiv^*(xz) =1\in \bar{\phiv^*}(y)$ and having the following property.
\begin{equation}\label{IIInon0}
\bar{\phiv^*}(x)\cap (\bar{\phiv^*}(y_{k})\cup \bar{\phiv^*}(y_{\ell})) \ne \emptyset
\end{equation}
Otherwise, by (\ref{III>D}), there exists $r\in \varphi (x)\cap \phibar(y_k)\cap \phibar(y_{\ell})$.  Choose a color $i\in  \bar{\varphi }(x) $.
Since at least one of colors $i$ and $r$ is missing at each of $x$, $y_k$ and $y_{\ell}$, we may assume  $P_{y_{k}}(i, r,\varphi )$ is disjoint with
 $P_{x}(i, r,\varphi )$.  Then, in coloring $\varphi /P_{y_{k}}(i, r,\varphi )$, color  $i$ is missing at both $x$ and $y_k$, giving a contradiction.

 By (\ref{IIInon0}), we may assume that there exists a color $i\in \phibar(x)\cap \phibar (y_{k})$.  Since $G$ is not edge-$\D$-colorable, $P_{x}(1, i, \varphi )=P_{y}(1, i, \varphi )$.  So, $P_{y_k}(1,i, \phiv)$ is disjoint with $P_x(1, i, \phiv)$.   If  $1\in \varphi (y_{k})$,  let $\varphi ^\prime=\varphi /P_{y_{k}}(1, i, \varphi )$. For coloring $\phiv^\prime$, we  have
 $1\in \bar{\phiv^\prime}(y_{k})$ and $\varphi^\prime (x)=\varphi(x), \varphi^\prime (y)=\varphi (y)$. Thus we can assume
$1\in \bar{\varphi }(y_{k})$. Let $\varphi^*$ be a coloring obtained from $\varphi$ by recoloring $yy_{k}$ with color $1$.  Then,  $\phiv^*(xz_k) =k \in \bar{\phiv^*}(y)$.   This completes the proof of {\bf III}.
\end{proof}



\section{Proof of Theorem \ref{main}}

Let $G$ be an edge-$\Delta$-critical graph with $n$ vertices and $m$ edges.
   Clearly, $\avd(G)=2m/n$.  We assume $\D\geq 56$. Let $q := \min\{ \frac{2\sqrt{2}(\Delta-1)-2}{2\sqrt{2}+1}, \frac{3}{4}\D-2\}$, that is, $q=\frac{2\sqrt{2}(\Delta-1)-2}{2\sqrt{2}+1}$ if $\D\ge 66$ and $q=\frac{3}{4}\D-2$ if $56\le \D\le 65$.
We initially assign to each vertex $x$ of $G$ a charge $M(x) = d(x)$ and redistribute the charge according to the following rule:
\begin{itemize}
\item {\bf Rule of Discharge:} each ($>q$)-vertex $y$ distributes its surplus
charge of $d(y)-q$ equally among all ($<q$)-neighbors of $y$.
\end{itemize}
Denote by $M^\prime (x)$ the resulting charge on each vertex $x$. Clearly, $\sum_{x\in V(G)} M^\prime (x)=\sum_{x\in V(G)} M(x)=2m$.
Let $X_1 =\{x\in V(G) \,:\,  d(x)\leq 3q-2\Delta\}$.
We show that $M^\prime(x) \ge 2 + 2(\D-q)$ for all vertices in $X_1$ and $M^\prime(x) \ge q$ for all other vertices, which gives $\avd(G) \ge q-(3q-2\D-2)\frac{|X_1|}{n}$. We then show that $|X_1|/n$ is small in order to complete our proof.

Since $q=\min\{\frac{2\sqrt{2}(\Delta-1)-2}{2\sqrt{2}+1}, \frac{3}{4}\D-2\}$ and $\Delta\ge 56$, we have $\frac{\D +2}2<q<\frac{3\D}{4}$. Thus $q > \D -q +2 > 3q -2\D$.

\begin{cla}\label{claim 1}
If $d(x)\leq \Delta-q+2$, then $M^\prime (x)\geq d(x)+2(\Delta-q)$. Consequently, $M^\prime (x)\geq d(x)+2(\Delta-q)$ for each $x\in X_1$.
\end{cla}
\begin{proof}
Let $y$ be an arbitrary neighbor of $x$.  Since  $2\Delta-d(x)-d(y)+2\ge \D -d(x) +2 \geq q$,
we have $\sigma_q(x, y) \ge \sigma(x, y)$. We will use lower bounds of $\sigma(x,y)$ to estimate $\sigma_q(x,y)$.
 Following the definition  $p_{min}(x)=  min_{v\in N(x)}\sigma(x,v)-\Delta+d(x)-1$ and  $p (x)  =  \min\{p_{min}(x), \lfloor\frac{d(x)}2 \rfloor-1\}$,
we have the following inequalities.
\begin{equation}\label{eq1}
1\leq d_{<q}(y)\leq d(y)-\sigma(x,y)\leq d(y)-(\Delta-d(x)+p(x)+1)
\end{equation}

By Lemma \ref{p},  $x$ has at least $d(x)-p(x)-1$ neighbors $y$ for which $\sigma(x,y)\geq \Delta-p(x)-1$, so for these neighbors $y$ the following inequalities hold.
\begin{equation}\label{eq2}
1\leq d_{<q}(y)\leq d(y)-\sigma(x,y)\leq d(y)-(\Delta-p(x)-1)
\end{equation}

We first consider the case $p(x)\geq 1$.
In this case, we have  $q\leq \Delta-d(x)+2\leq \Delta-d(x)+p+1$.
Since $\frac{d(y)-a}{d(y)-b}$ with $a\leq b$ is a decreasing function of $d(y)$,
for each $y\in N(x)$,  $x$ receives charge at least
$$
\frac{ d(y)-q}{d(y)-(\Delta-d(x)+p(x)+1)}\geq \frac{\Delta-q}{d(x)-p(x)-1},
$$
And there are at least $d(x)-p(x)-1$ neighbors $y$ of $x$ giving $x$ at least
$$
\frac{ d(y)-q}{d(y)-(\Delta-p(x)-1)}\geq \frac{\Delta-q}{p(x)+1},
$$
where the inequality holds because $q\leq \Delta-d(x)+2\leq \Delta-p(x)-1$ as $1\leq p(x)\leq \lfloor\frac{d(x)}{2}\rfloor-1$.
Thus $x$ receives at least
$$
  (d(x)-p(x)-1)\frac{\Delta-q}{p(x)+1}+(p(x)+1)\frac{\Delta-q}{d(x)-p(x)-1}=(\theta+\theta^{-1})(\Delta-q)\geq 2(\Delta-q),
$$
where $\theta=\frac{d(x)-p(x)-1}{p(x)+1}$. It follows that
$M^\prime(x)\geq M(x)+2(\Delta-q)=d(x)+2(\Delta-q)$.

We now consider the case $p (x)  =  \min \{p_{min}(x), \lfloor\frac{d(x)}2 \rfloor -1\} =0$. If $d(x)=2$, then by (\ref{eq1}) for every neighbor $y$ of $x$ we have $d_{<q}(y)=1$ and $d(y)=\Delta$, thus $M^\prime(x)\geq M(x)+2(\Delta-q)=d(x)+2(\Delta-q)$. If $d(x)\geq 3$, then by (\ref{eq2}) for at least $d(x)-1$ neighbors $y$ of $x$,  we have $d_{<q}(y)=1$ and $d(y)=\Delta$.  Thus $M^\prime(x)\geq M(x)+(d(x)-1)(\Delta-q)\geq d(x)+2(\Delta-q)$.
\end{proof}

\begin{cla}\label{claim 2}
For each $x\in V(G)-X_1$, $M^\prime (x)\geq q$ holds.
\end{cla}
\begin{proof}
Let $x\in V(G) -X_1$, i.e., $d(x) > 3q-2\D$.  If $d(x)\geq q$, then $M^\prime(x)=M(x)-\frac{d(x)-q}{d_{<q}(x)}d_{<q}(x)=q.$
If $3q-2\Delta<d(x)\leq \Delta-q+2$, then by Claim \ref{claim 1}, we have $M^\prime(x)\geq d(x)+2(\Delta-q)>q$.
So we only need to consider the case $\Delta-q+2<d(x)<q$.

Since $G$ is edge-$\D$-critical and $xy\in E(G)$, there exists a coloring $\phiv\in \mathcal{C}^{\D}(G-xy)$. Let $Z_q=\{z\in N(x)\,:\, d(z)>q\}$, $Z_y=\{z\in N(x)\backslash \{y\}\,:\, \phiv(xz)\in \phibar(y)\}$ and $Z_q^*=Z_q\cap Z_y$. Clearly, for each $z\in Z_q$, $x$ receives at least $\frac{d(z)-q}{d(z)-\sigma_{q}(x,z)}$ charge. Thus $M^\prime (x) \geq d(x)+\sum_{z\in Z_q}\frac{d(z)-q}{d(z)-\sigma_{q}(x,z)}$.
We consider the following three cases to complete the proof.

{\flushleft \textbf{Case 1.}} $\Delta-q+2<d(x)<q$ and $x$ has a neighbor $y$ such that $d(y)\leq q$.

By Lemma \ref{lemfact} (3), for each vertex $z\in Z_q^*$ we have
\[
\sigma_{q}(x,z)  \geq  2\Delta-d(x)-d(y)+1-\lfloor\frac{d(x)+d(y)+d(z)-2\Delta-2}{\Delta-q}\rfloor
\ge 2\D -d(x) -d(y),
\]
where we used the inequality $\lfloor\frac{d(x)+d(y)+d(z)-2\Delta-2}{\Delta-q}\rfloor\leq 1$ following
 $d(x)<q, d(y)\leq q$, $d(z)\leq \D$ and $q <\frac{3}{4}\D$. Thus  $\sigma_{q}(x,z)  \geq  2\Delta-d(x)-d(y)$. So, $M^\prime (x) \geq M(x)+\sum_{z\in Z_q^*}\frac{d(z)-q}{d(z)-(2\Delta-d(x)-d(y))}$.
 By Lemma \ref{lemfact} (2), we have $\sum_{z\in Z_q^*}(d(z)-q)\geq (\Delta-d(y)+1)(\Delta-q)-d(x)-d(y)+\Delta+2$. Thus
\begin{eqnarray*}
\sum_{z\in Z_q^*}\frac{d(z)-q}{d(z)-(2\Delta-d(x)-d(y))}  &\ge&  \frac{(\Delta-q)(\Delta-d(y)+1)-(d(x)+d(y)-\Delta-2)}{d(x)+d(y)-\Delta}\\
&\ge& \frac{(\Delta-q)(\Delta-d(y)+1)+2}{d(x)+d(y)-\Delta}-1.
\end{eqnarray*}
So, we have the following inequality.
\begin{eqnarray*}
  M^\prime (x) &\geq & M(x)+\sum_{z\in Z_q^*}\frac{d(z)-q}{d(z)-(2\Delta-d(x)-d(y))} \\
   &\geq & d(x)+\frac{(\Delta-q)(\Delta-d(y)+1)+2}{d(x)+d(y)-\Delta}-1 \\
  &\geq & d(x)+q-\D+\frac{(\Delta-q)(\Delta-q+1)+2}{d(x)+q-\Delta}-1-q+\D\\
   &\geq & 2\sqrt{(\Delta-q)(\Delta-q+1)+2}+\Delta-q-1\\
   &\geq & 3(\Delta-q) \geq  q.
\end{eqnarray*}

{\flushleft \textbf{Case 2.}} $2(\Delta-q)-4< d(x)<q$ and $d(y)> q$ for every neighbor $y$ of $x$.

Let  $y\in N(x)$ such that $d(y):=\min \{d(u)\,:\, u\in N(x)\}$.
By Lemma \ref{lemfact} (3), for each vertex $z\in Z_q^*$ we have
\begin{eqnarray*}
\sigma_{q}(x,z)  & \geq & 2\Delta-d(x)-d(y)+1-\left\lfloor\frac{d(x)+d(y)+d(z)-2\Delta-2}{\Delta-q}\right\rfloor
\\
& \ge & 2\D -d(x) -d(y) -1,
\end{eqnarray*}
where we used the inequality
$\lfloor\frac{d(x)+d(y)+d(z)-2\Delta-2}{\Delta-q}\rfloor\leq 2$ when
$d(x)<q$ and $q <\frac{3}{4}\D$.
By Lemma \ref{lemfact} (2), we have
\begin{eqnarray*}
\sum_{z\in Z_q^*}\frac{d(z)-q}{d(z)-\sigma_q(x,y)} & \ge &  \frac{ (\Delta-d(y)+1)(\Delta-q)-d(x)-d(y)+\Delta+2}{d(x)+d(y)-\Delta+1}\\
& = & q-\D-1+\frac{(\Delta-q)(d(x)+2)+3}{d(x)+d(y)-\Delta+1}.
\end{eqnarray*}

By Lemma \ref{VAL}, for each neighbor $u$ of $x$ we have $\sigma_{\D}(x,u)\ge \Delta-d(x)+1$. Since $d(u) \ge d(y)$ for each $u\in N(x)$ and $q\geq \Delta-d(x)+1$, we have
 \[
 \frac{d(u)-q}{d(u)-(\Delta-d(x)+1)}\geq \frac{d(y)-q}{d(y)-(\Delta-d(x)+1)}.
 \]
 So,
 \begin{eqnarray*}
 && \sum_{u\in N(x)\backslash Z_q^*}\frac{d(u)-q}{d(u)-(\Delta-d(x)+1)}
  \ge   |N(x) \setminus Z_y| \cdot \frac{d(y)-q}{d(y)-(\Delta-d(x)+1)}\\
 && =  \frac{(d(x)-(\D-d(y)+1)) (d(y)-q)}{d(y)-(\Delta-d(x)+1)}=d(y)-q.
 \end{eqnarray*}
    Thus
\begin{eqnarray*}
  M^\prime (x) &\geq & d(x)+q-\D-1+\frac{(\Delta-q)(d(x)+2)+3}{d(x)+d(y)-\Delta+1}
  + d(y)-q\\
   &= & d(x)+d(y)-\Delta+1+\frac{(\Delta-q)(d(x)+2)+3}{d(x)+d(y)-\Delta+1}-2\\
   &\geq & 2\sqrt{(\Delta-q)(d(x)+2)+3}-2\\
   &>&  2\sqrt{2(\Delta-q)(\Delta-q-1)}-2\geq  q.
\end{eqnarray*}

{\flushleft \textbf{Case 3.}} $\Delta-q+2< d(x)\leq 2(\Delta-q)-4$ and $d(y)> q$ for every neighbor $y$ of $x$.

Since $\D\ge 56$, we have $\Delta-q+2< 2(\Delta-q)-4$, so this case occurs.
 Since the notation $p(x,q)$ will be used heavily in this proof, we let $p^\prime := p(x,q)$ for convenience.
 So, $p^\prime=  \min \{\ p_{min}(x,q), \lfloor\frac{d(x)}{2}\rfloor-3\ \}$, where
$p_{min}(x,q) := \min_{y\in N(x)}\sigma_q(x,y)-\Delta+d(x)-1 $.
Following this definition, for every $y\in Z_q$,  $\sigma_q(x,y)\geq \Delta-d(x)+p^\prime+1$, which in turn gives
\[
\frac{d(y) - q}{d(y) -\sigma_q(x,y)} \ge \frac {d(y) - q}{d(y) - (\D -d(x) + p^\prime +1)}.
\]
So, if $q \le \D -d(x) +p^\prime +1$, then
\[
\frac{d(y) - q}{d(y) -\sigma_q(x,y)} \ge \frac {\D - q}{d(x) -p^\prime -1}.
\]

By Lemma \ref{pp}, $x$ has at least $d(x)-p^\prime-3$ neighbors $y$ for which $\sigma_q(x,y)\geq\Delta-p^\prime-5$.  For such neighbors $y$, since $q \le \D -\frac{d(x)}2 -2 \le \D -p^\prime -5$, we have
 \[
 \frac{d(y) - q}{d(y) -\sigma_q(x,y)} \ge \frac{d(y) -q}{d(y) - (\D -p^\prime -5)} \ge \frac{\D -q} {p^\prime +5}.
 \]
 If $q\leq \Delta-d(x)+p^\prime+1$, then
 \begin{eqnarray*}
 M^\prime(x) & \geq &
 d(x)+(d(x)-p^\prime-3)\frac{\D -q}{p^\prime +5}+(p^\prime+3)\frac{\D-q}{d(x)-p^\prime-1}\\
 & \ge & \D -q +2 + (\D -q) (2 -\frac{8(d(x)+2)}{(\Delta-q+2)(d(x)+4)}) \ge q,
 \end{eqnarray*}
where we used the inequality $\theta+\theta^{-1}\ge 2$ ($\theta=\frac{d(x)-p^\prime-3}{p^\prime+5}$) to show the following
 $$\frac{d(x)-p^\prime-3}{p^\prime+5}+\frac{p^\prime+3}{d(x)-p^\prime-1}>2-\frac{2d(x)+4}{(\frac{\Delta-q+2}{2})(\frac{d(x)}{2}+2)}=2-\frac{8(d(x)+2)}{(\Delta-q+2)(d(x)+4)}.$$

Suppose $q> \Delta-d(x)+p^\prime+1$, i.e.,  $p^\prime< d(x)+q-\Delta-1$. So,
 $\frac{d(x)-p^\prime-3}{p^\prime+5}>\frac{\D-q-2}{d(x)+q-\D+4}$, which gives
\begin{eqnarray*}
M^\prime(x) & \geq &  d(x) +\frac{\D-q-2}{d(x)+q-\D+4}(\D-q)\\
& = & (d(x)+q-\Delta+4)+\frac{(\Delta-q)(\Delta-q-2)}{d(x)+q-\D+4}-(q-\Delta+4)\\
& \ge &  2\sqrt{(\Delta-q)(\Delta-q-2)}+\Delta-q-4\geq  3(\Delta-q)-8 \geq  q.
\end{eqnarray*}
\end{proof}

\begin{cla}\label{claim 3}
$d(y)> q$ for each $y\in N(X_1)$ and $|N(X_1)|\geq 2|X_1|$ where $N(X_1)=\cup_{x\in X_1}N(x)$.
\end{cla}
\begin{proof}
Since $G$ is edge-$\D$-critical, for each edge $xy\in E(G)$ we have
$d(x) + d(y) \ge \D +2$.
Since $q<\frac{3}{4}\D$ and $d(x)\le 3q-2\D$ for each $x\in X_1$, we have $d(y)\geq \Delta+2-(3q-2\Delta)>q$ for each $y\in N(x)$.
Thus the vertices in $N(X_1)$ does not receive charges from any other vertices.
As the vertices in $X_1$ receive charges only from the vertices in $N(X_1)$, we have
\begin{equation}\label{eq5}
\sum_{x\in X_1}M^\prime (x)+\sum_{y\in N(X_1)}M^\prime (y)\leq \sum_{x\in X_1}M(x)+\sum_{y\in N(X_1)}M(y)\leq \sum_{x\in X_1}d(x)+\Delta|N(X_1)|.
\end{equation}
Also, by Claims \ref{claim 1} and \ref{claim 2}, we have $M^\prime (x)\geq d(x)+2(\Delta-q)$ for each $x\in X_1$ and $M^\prime (y)\geq q$ for each $y\in N(X_1)$. Thus we have
\begin{equation}\label{eq6}
\sum_{x\in X_1}M^\prime (x)+\sum_{y\in N(X_1)}M^\prime (y)\geq \sum_{x\in X_1}d(x)+2(\Delta-q)|X_1|+q|N(X_1)|.
\end{equation}
Combining (\ref{eq5}) with (\ref{eq6}), we have $|N(X_1)|\geq 2|X_1|$.
\end{proof}

For each edge $xy\in E(G)$ and $\phiv\in \mathcal{C}^{\D}(G-xy)$, let $Y(x,\phiv)=\{w\in N(y)\backslash \{x\}\, :\, \phiv(yw)\in \bar{\phiv}(x)\}$ , $Y^1(x,\phiv)=Y(x,\phiv)\cap N(X_1)$ and $Y^2(x,\phiv)=Y(x,\phiv)-(X_1\cup N(X_1))$.
Clearly, $|Y(x,\phiv)|=\D-d(x)+1$.
Note that with respect to the coloring $\phiv$, $\{x,y\}\cup Y(x,\phiv)$ forms a Vizing fan, so it is elementary.
\begin{cla}\label{claim 4}
For each $y\in N(X_1)$ and $x\in N(y)\cap X_1$, $|Y^2(x,\phiv)|\ge \Delta-2d(x)+3$.
\end{cla}
\begin{proof}
Recall that $\{x,y\}\cup Y(x,\phiv)$ is elementary. Then for each vertex in $w\in Y(x,\phiv)$ we have $|\bar{\phiv}(w)|+|\bar{\phiv}(x)|+|\bar{\phiv}(y)|\le \D$, it follows that $d(w)\ge |\bar{\phiv}(x)|+|\bar{\phiv}(y)|>3q-2\D$. Thus we have $Y(x,\phiv)\cap X_1=\emptyset$. If $|Y^1(x,\phiv)|\le d(x)-2$, then $|Y^2(x,\phiv)|=|Y(x,\phiv)-Y^1(x,\phiv)-(Y(x,\phiv)\cap X_1)|\ge \D-d(x)+1-(d(x)-2)\ge \Delta-2d(x)+3$.
So the Claim \ref{claim 4} is equivalent to show that $|Y^1(x,\phiv)|\le d(x)-2$.

\textbf{Subclaim \ref{claim 4}.1.} If $w\in Y^1(x,\phiv)$, then for each neighbor $z$ of $w$ in $X_1$, we have $\varphi(wz)\in \varphi(x)\cap \varphi(y)$.

\begin{proof} If there exists a neighbor of $w$ in $X_1$, say $z$, such that $\varphi(wz)\notin \varphi(x)\cap \varphi(y)$. Then $\varphi(wz)\in \bar{\varphi}(x)\cup \bar{\varphi}(y)$. Thus $\{x,y,w,z\}$ forms a Kierstead path. By Lemma \ref{p4}, we have $|\bar{\varphi}(z)\cap (\bar{\varphi}(x)\cup \bar{\varphi}(y))|\leq 1$, it follows that $d(z)\geq (\Delta-d(x)+1)+(\Delta-d(y)+1)-1> 3q-2\Delta$, this contradicts with the fact that $z\in X_1$. So Subclaim \ref{claim 4}.1. holds.
\end{proof}

For each color $j\in \varphi(x)\cap \varphi(y)$, set $Y_j=\{w\in Y^1(x,\phiv)\, :\, j\in \varphi(w)\}$ and $Z_j=\{z\in X_1: \mbox{\ there exists a vertex \ } w\in Y_j \mbox{\ such that \ } \varphi(wz)=j\}$. Clearly, $\sum_{j\in \varphi(x)\cap \varphi(y)}|Z_j|\ge |Y^1(x,\phiv)|$. Since $|\varphi(x)\cap \varphi(y)|=\D-(\D-d(x)+1)-(\D-d(y)+1)\le d(x)-2$, to show that $|Y^1(x,\phiv)|\le d(x)-2$, we only need to prove that $|Z_j|\leq 1$ for each $j$.
Let $|Z_j|=t$ and $Z_j=\{z_{\alpha_1},\cdots,z_{\alpha_t}\}$, where for each $z_{\alpha_i}$ there exists $y_{\alpha_i}\in Y_j$ such that $\phiv(yy_{\alpha_i})=\alpha_i$ and $\phiv(y_{\alpha_i}z_{\alpha_i})=j$. Clearly, $\alpha_i\in \bar{\varphi}(x)$ for each $1\le i\le t$.

\textbf{Subclaim \ref{claim 4}.2.} Let $k$ be a color in $\bar{\varphi}(x)$. Then the followings hold.

(1) For each $k\notin \{\alpha_1,\cdots,\alpha_t\}$, at least $t-1$ vertices of $Z_j$ have the color $k$.

(2) For each $k\in \{\alpha_1,\cdots,\alpha_t\}$, at least $t-2$ vertices of $Z_j$ have the color $k$.
\begin{proof}
First suppose that $k\notin \{\alpha_1,\cdots,\alpha_t\}$. We consider the path $P_x(j,k,\varphi)$, and $w$ is the other end vertex of this path. We will show that the color $k$ seen by each vertex in $Z_j\backslash\{w\}$. For otherwise, we assume $k\notin \varphi(z)$ for some $z\in Z_j\backslash\{w\}$, say $z=z_{\alpha_1}$, then $P_{z_{\alpha_1}}(j,k,\varphi)$ is disjoint from $P_x(j,k,\varphi)$, thus let $\varphi^\prime =\varphi/P_{z_{\alpha_1}}(j,k,\varphi)$ be the new coloring which $\varphi^\prime(yy_{\alpha_1})=\varphi(yy_{\alpha_1})\in \bar{\varphi^\prime}(x)$ and $\varphi^\prime(y_{\alpha_1}z_{\alpha_1})=k \in \bar{\varphi^\prime}(x)$. Thus $\{x,y,y_{\alpha_1},z_{\alpha_1}\}$ forms a Kierstead path. So by Lemma \ref{p4} we have $|\bar{\varphi}(z_{\alpha_1})\cap (\bar{\varphi}(x)\cup \bar{\varphi}(y))|\leq 1$, it follows that $d(z_{\alpha_1})\geq (\Delta-d(x)+1)+(\Delta-d(y)+1)-1> 3q-2\Delta$ as $d(x)\le 3q-2\D$, this contradicts with the fact that $z_{\alpha_1}\in X_1$.

Then suppose that $k\in \{\alpha_1,\cdots,\alpha_t\}$. We may assume that $k=\alpha_t$. Clearly, $k\notin \{\alpha_1,\cdots,\alpha_{t-1}\}$. Let $Z^\prime_j=Z_j-\{z_k\}$. By (1), we have at least $|Z^\prime_j|-1$ vertices of $Z^\prime_j$ has the color $k$, that is, at least $t-2$ vertices of $Z_j$ has the color $k$.
\end{proof}
By Subclaim \ref{claim 4}.2 and the definition of $X_1$, we have
$$(\Delta-d(x)+1-t)(t-1)+t(t-2)\leq \sum_{z\in Z_j}d(z)\leq t(3q-2\Delta).$$
Since $d(x)\le 3q-2\D$ and $q<\frac{3}{4}\D$, we have
\begin{equation*}
 t \leq  \frac{\Delta-d(x)+1}{3\Delta-3q-d(x)} \leq 1+ \frac{3q-2\Delta+1}{5\Delta-6q}< 2.
\end{equation*}
Since $t$ is an integer, we have $t\leq 1$. Then Claim \ref{claim 4} holds.
\end{proof}

Let $c$ be a positive integer, set $Z_1(c)=\{z\in V(G)-(X_1\cup N(X_1))\,:\,d(z)\geq \Delta-c\}$ and $Z_2(c)=\{z\in V(G)-(X_1\cup N(X_1))\,:\,d(z)< \Delta-c\}$.

\begin{cla}\label{claim 5}
$|Z_1(c)|\geq \frac{(5c+2)\Delta-(6c+3)q+3c+2}{c\Delta}|N(X_1)|$.
\end{cla}
\begin{proof}
For each $y\in N(X_1)$, $x\in N(y)\cap X_1$ and $\phiv\in \mathcal{C}^{\D}(G-xy)$, let $Y_{<c}=\{z\in Y^2(x,\phiv)\,:\,d(z)<\D-c\}$.
Since $\{x,y\}\cup Y(x,\phiv)$ is elementary and $Y^2(x,\phiv)\subseteq Y(x,\phiv)$, we have
$\D-d(x)+1+\D-d(y)+1+c|Y_{<c}|< \sum_{v\in \{x,y\}\cup Y(x,\phiv)}\phibar(v)\le \D$.
Thus $|Y_{<c}|<\frac{d(x)-2}{c}$.
By Claim \ref{claim 4}, we have $|Y^2(x,\phiv)-Y_{<c}|> \D-2d(x)+3-\frac{d(x)-2}{c}$, that is, for each $y\in N(X_1)$ we have $d_{Z_1(c)}(y)\ge \D-2d(x)+3-\frac{d(x)-2}{c}$.
Hence,
$$(\D-2d(x)+3-\frac{d(x)-2}{c})|N(X_1)|\le |E(N(X_1),Z_1(c))|\le \D|Z_1(c)|,$$
where $E(N(X_1),Z_1(c))$ are the edges with one vertex in $N(X_1)$ and the other endvertex in $Z_1(c)$.
Since $d(x)\le 3q-2\D$, solving the above inequalities we have
$$|Z_1(c)|\geq \frac{(5c+2)\Delta-(6c+3)q+3c+2}{c\Delta}|N(X_1)|.$$
\end{proof}

By Claims \ref{claim 1} and \ref{claim 2}, we have
\[
M^\prime(x) \geq
\begin{cases}
 2+2(\D-q)\qquad x\in X_1,\\
 q\qquad\qquad\qquad\quad x\in V(G)-X_1.
 \end{cases}
 \]
And by the definitions of $Z_1(c)$ and $Z_2(c)$, we get the following two lower bounds of $\sum_{x\in V(G)}M^\prime(x)$.
\begin{eqnarray*}
  b_1 &=& (2+2(\Delta-q))|X_1|+q|N(X_1)|+(\Delta-c)|Z_1(c)|+(3q-2\Delta)|Z_2(c)| \\
  b_2 &=& (2+2(\Delta-q))|X_1|+(n-|X_1|)q
\end{eqnarray*}
We now divide into a few cases to estimate the lower bound of $\max\{b_1,b_2\}$.

First we consider the case $\Delta-q-c> 0$.
For fixed value $|X_1|+|N(X_1)|$, $|Z_1(c)|+|Z_2(c)|$ is a constant.
Since $\D-c> q>3q-2\D$, $\max\{(\Delta-c)|Z_1(c)|+(3q-2\Delta)|Z_2(c)|, q(|Z_1(c)|+|Z_2(c)|)\}$ takes minimum when
$(\Delta-c)|Z_1(c)|+(3q-2\Delta)|Z_2(c)|=q(|Z_1(c)|+|Z_2(c)|)$, that is $|Z_2(c)|= (\frac{1}{2}-\frac{c}{2\Delta-2q})|Z_1(c)|$.
So by Claims \ref{claim 3} and \ref{claim 5}, $|Z_1(c)|+|Z_2(c)|\ge (3-\frac{c}{\Delta-q})\frac{(5c+2)\Delta-(6c+3)q+3c+2}{c\Delta}|X_1|$.
Let $f(c)=(3-\frac{c}{\Delta-q})\frac{(5c+2)\Delta-(6c+3)q+3c+2}{c\Delta}$.
So $n = |X_1|+|N(X_1)|+|Z_1(c)|+|Z_2(c)|\ge (3+f(c))|X_1|$.
Hence,
$\sum_{x\in V(G)}M^\prime(x)\geq \max\{b_1,b_2\}\ge qn+(2+2\Delta-3q)|X_1|
\ge (q+\frac{2+2\D-3q}{3+f(c)})n.$
So $\bar{d}(G)\ge q+\frac{2+2\D-3q}{3+f(c)}$.

Let $q^*=\frac{2\sqrt{2}\Delta}{2\sqrt{2}+1}$ and $a=1+\frac{1}{2\sqrt{2}+1}$ if $q=\frac{2\sqrt{2}(\Delta-1)-2}{2\sqrt{2}+1}$, and $q^*=\frac{3\D}{4}$, $a=2$ if $q=\frac{3}{4}\D-2$. So $q=q^*-a$ and we have
\begin{eqnarray*}
\frac{2+2\D-3q}{3+f(c)}&=&\frac{c\D(2+2\D-3q)}{3c\D+(3-\frac{c}{\D-q})((5c+2)\Delta-(6c+3)q+3c+2)}\\
&=&\frac{(2c\D-3cq^*)\D+f_1f_2\D-f_1f_2\D+2c\D+3ca\D}{(18c+6)\D-(18c+9)q^*+f_2}\\
&=& f_1\D+\frac{-f_1f_2\D+2c\D+3ca\D}{(18c+6)\D-(18c+9)q^*+f_2},
\end{eqnarray*}
where $f_1=\frac{2c\D-3cq^*}{(18c+6)\D-(18c+9)q^*}$ and $f_2=9c+6+(18c+9)a-\frac{(5c^2+2c)\Delta-(6c^2+3c)q+3c^2+2c}{\D-q}$.

Clearly, $f_1=\frac{2c\D-3cq^*}{(18c+6)\D-(18c+9)q^*}=\frac{2\D-3q^*}{18\D-18q^*+\frac{6\D-9q^*}{c}}$ is an increasing function of $c$.
To make $f_1$ as large as possible when $\D\ge l$, where $l$ is a positive integer,
we choose $c$ such that $c=\min\{\lfloor\D-q\rfloor\,:\,\D\ge l\}$.
If $l$ is large enough, $c$ is large too and we can see that the value of $f_1$ will approximate to $\frac{2\D-3q^*}{18\D-18q^*}$ and $\bar{d}(G)$ will approximate to $0.69277\D$.
Note that $q=\frac{2\sqrt{2}(\Delta-1)-2}{2\sqrt{2}+1}$ if $\D\ge 66$. Let $l=66$.
Then we have $c=18$. So we have $\D-q-c> 0$ if $\D\ge 65$.
Plugging $c=18$ and the value of $q^*$ into $f_2$ and $\frac{-f_1f_2\D+2c\D+3ca\D}{(18c+6)\D-(18c+9)q^*+f_2}$, we have $\frac{-f_1f_2\D+2c\D+3ca\D}{(18c+6)\D-(18c+9)q^*+f_2}>0$ and $f_2<0$. Thus
\begin{equation}\label{eqf}
  \frac{2+2\D-3q}{3+f(c)}\ge f_1\D+\frac{2c\D+3ca\D-f_2f_1\D}{(18c+6)\D-(18c+9)q^*}.
\end{equation}
If $\D\ge 66$, then $q=\frac{2\sqrt{2}(\Delta-1)-2}{2\sqrt{2}+1}$. Plugging $c=18$ into the inequality (\ref{eqf}), we get
$f_1\ge -0.04638$, $f_2\ge -244.43905$ and
\begin{eqnarray*}
\frac{2+2\D-3q}{3+f(c)}&\ge& -0.04638\D+1.10463.
 \end{eqnarray*}
Thus
$$\bar{d}(G)\ge q-0.04638\D+1.10463\ge 0.69241\D-0.15658.$$
If $\D=65$, then $q=\frac{3}{4}\D-2$. Plugging $c=18$ and $\D=65$ into the inequality (\ref{eqf}), we get $f_1\ge -0.05608$, $f_2\ge -1.15069$ and
$\frac{2+2\D-3q}{3+f(c)}\ge -0.05608\D+1.79358$. It follows that
$$\bar{d}(G)\ge q-0.05608\D+1.79358\ge 0.69392\D-0.20642.$$

Now we consider the case $\D-q-c\le0$. It is easy to see that $b_2>b_1$ and $\D\le 64$. Thus
$\sum_{x\in V(G)}M^\prime(x)\ge qn-(3q-2\D-2)|X_1|$. So $\bar{d}(G)\ge q-(3q-2\D-2)\frac{|X_1|}{n}$.
By Claims \ref{claim 3} and \ref{claim 5}, we have
\begin{equation*}
n\ge |X_1|+|N(X_1)|+|Z_1(c)|\geq (3+f^\prime(c))|X_1|,
\end{equation*}
where $f^\prime(c)=2\frac{(5c+2)\Delta-(6c+3)q+3c+2}{c\Delta}$.

Plugging $c=18$ and $q=\frac{3}{4}\D-2$ into $f^\prime(c)$, we have $f^\prime(c)=\frac{8.75\D+278}{9\D}$. Since
$\frac{|X_1|}{n}\le \frac{1}{3+f^\prime(c)}$,
we have $(3q-2\D-2)\frac{|X_1|}{n}\le \frac{2.25\D^2-72\D}{35.75\D+278}=\frac{9\D}{143}+\frac{695.94484}{35.75\D+278}-2.50339$,
thus $\bar{d}(G)\geq  \frac{393\D}{572}-\frac{695.94484}{35.75\D+278}+0.50339$. It is easy to check that if $\D\ge 56$ then $\bar{d}(G)\ge \frac{2}{3}(\D+2)$, which improve Woodall's result in \cite{w2}. If $\D\ge 56$, we have
$$\bar{d}(G)\ge  0.68706\D+0.19815.$$

Hence,
\[
\bar{d}(G) \geq
\begin{cases}
 0.69241\D-0.15658 \quad\,\: \mbox{ if } \Delta\geq 66, \\
 0.69392\D-0.20642\quad\;\,\mbox{ if } \Delta=65, \mbox{ and  } \\
 0.68706\D+0.19815\quad\! \quad\mbox{if } 56\leq \Delta\leq64.
 \end{cases}
 \]
This completes the proof of Theorem \ref{main}.

\end{document}